\documentclass{amsart}
\usepackage{amssymb, amsmath, verbatim}
\usepackage[all]{xy}
\newcommand{\itz}{\begin{itemize}}
\newcommand{\zti}{\end{itemize}}
\newcommand{\enu}{\begin{enumerate}}
\newcommand{\une}{\end{enumerate}}
\newcommand{\ie}{{\it i.e.},}

\newcommand{\om}{\omega}

\newcommand{\eps}{\varepsilon}
\newcommand{\naturals}{\mathbb{N}}

\newcommand{\inv}[1]{#1^{-1}}

\newcommand{\trileq}{\trianglelefteq}

\newcommand{\mc}[1]{\mathcal{#1}}

\newcommand{\sark}{\v Sarkovski\u\i}
\newcommand{\tand}{\text{ and }}
\newcommand{\tor}{\text{ or }}
\newcommand{\tcase}{\text{case }}
\newcommand{\self}[2]{#1\colon#2\rightarrow #2}

\newcommand{\bo}{\trileq} 

\newcommand{\divides}{\mid}
\newcommand{\notdivides}{\nmid}
\newcommand{\covers}[1]{\supset_{#1}}
\newcommand{\closure}[1]{\overline{#1}}
\DeclareMathOperator{\diam}{diam}
\begin{document}

\theoremstyle{plain}
\newtheorem{theorem}{Theorem}
\newtheorem{lemma}{Lemma}
\newtheorem{proposition}{Proposition}
\newtheorem{corollary}{Corollary}

\theoremstyle{definition}
\newtheorem{definition}{Definition}
\newtheorem{example}{Example}
\newtheorem{exercise}{Exercise}

\theoremstyle{remark}
\newtheorem{remark}{Remark}
\newtheorem{conjecture}{Conjecture}
\newtheorem{problem}{Problem}
\newtheorem*{claim}{Claim}

\title{Chaos and periodicity on star graphs}
\author{Jorge L. Guerrero}
\email{agntp123@dusty.tamiu.edu}
\address{
Dept. of Mathematics and Physics, 
Texas A{\&}M International University, 
5201 University Blvd., Laredo, TX 78041, USA
}

\author{David Milovich}
\email{david.milovich@tamiu.edu}
\urladdr{http://dkmj.org/academic}
\address{
Dept. of Mathematics and Physics, 
Texas A{\&}M International University, 
5201 University Blvd., Laredo, TX 78041, USA
}
\date{\today}
\begin{abstract}
  For a continuous self-map of a star graph
  to be Li-Yorke chaotic
  and to have full periodicity,
  we prove some new sufficient conditions
  on the orbit of the center.
\end{abstract}

\maketitle

\section{Introduction and main results}

By the \emph{$n$-od}, we mean a topological space
$X_n$ that is homeomorphic to the star graph of order $n$,
also known as the \emph{$n$-star} $S_n$.
The \emph{triod} is $X_3$, which is also known
as the simple dendrite or as $Y$.
The \emph{center} of $X_n$ is its vertex of order $n$,
which we denote by $o$.
A \emph{proper branch} of $X_n$ is
a connected component of $X_n\setminus\{o\}$;
fix an enumeration $\beta_1,\ldots,\beta_n$
of these proper branches.
A \emph{branch} of $X_n$ is the closure of a proper branch.

The original motivation for our results
was to find a new generalization to the triod
of Li and Yorke's ``Period three implies chaos''
for the interval, and to avoid the
uninteresting case of maps $\self{f}{X_3}$
of the form $\iota\circ g\circ r$ where
$r$ is a retraction of $X_3$ to $[0,1]$,
$\iota$ is its unique right inverse, and
$\self{g}{[0,1]}$.
As a special case of Corollary~\ref{fromperiod}
below, we meet this goal: if $\self{f}{X_3}$ and
the orbit of $o$ intersects each proper branch exactly once,
then $f$ is Li-Yorke chaotic and has full periodicity.
(We assume all maps are continuous.)

\begin{theorem}\label{allperiods}
  If $\self{f}{X_n}$ and
  $f^3(o)$ is not on the same branch as $f(o)$,
  then $f$ has points of all periods.
\end{theorem}

\begin{theorem}\label{scramble}
  If $\self{f}{X_n}$ and
  $f^3(o)$ is not on the same branch as $f(o)$,
  then $f$ scrambles an uncountable set.
\end{theorem}

Here $S\subset X_n$ is \emph{scrambled}~\cite{ly}
by $\self{f}{X_n}$ if, for all distinct $p,q\in S$,
\[\liminf_{i\rightarrow\infty} d(f^i(p),f^i(q))=0
<\limsup_{i\rightarrow\infty} d(f^i(p),f^i(q))\]
where $d$ is a metric compatible with the topology of $X_n$.
Because $X_n$ is compact, whether $S$ is scrambled or not
does not depend on $d$:
the identity map from $(X_n,d_1)$ to $(X_n,d_2)$
is uniformly continuous for all pairs $(d_1,d_2)$
of compatible metrics.
$\self{f}{X_n}$ is called \emph{Li-Yorke} chaotic
if it scrambles an uncountable set.

Theorems~\ref{allperiods} and \ref{scramble}
are proved in section~\ref{proofs}.
The proof of Theorem~\ref{scramble}
mainly uses ideas from Li and Yorke's 
scrambled set construction~\cite{ly}.
The proof of Theorem~\ref{allperiods} leans more
heavily on techniques involving
``basic intervals'' similar to Baldwin's~\cite{b}.

\begin{corollary}\label{fromperiod}
  If $n\ge 2$, $\self{f}{X_n}$, and the orbit of $o$
  has size $n+1$ and intersects every proper branch,
  then $f$ is Li-Yorke chaotic and has full periodicity.
\end{corollary}

For comparison, Alsed\`a and Moreno~\cite{am} proved
that, for an arbitrary $\self{f}{X_3}$, if
the periodicity of $f$ does not contain $\{2,3,4,5,7\}$,
then $f$ may not have full periodicity.
(By periodicity of $f$, we mean
the set of all $f$-periods of points in $X_3$.)
If $n=3$ in Corollary~\ref{fromperiod}, then
period 4 for an ``interesting'' orbit of the center
implies full periodicty. In section 2, we 
compare Corollary~\ref{fromperiod} to Baldwin's
characterizations of periodicity sets of self maps of $X_n$.

In section~\ref{extensions}, we show that
``$n+2$'' can replace ``$n+1$'' in Corollary~\ref{fromperiod}
at the cost of assuming $n\ge 3$ and
weakening ``full periodicity'' to ``all periods except 3.''
We show by example that period 3 can indeed be avoided.
We also give an example showing that 
all odd periods $\ge 3$ can be avoided
if ``$n+3$'' replaces ``$n+1$.''

\section{Relation to Baldwin's characterization}
Baldwin~\cite{b} defines, given a topological space $X$,
a preorder (\ie\ transitive and reflexive relation)
$\le_X$ of $\naturals$ by $p\le_X q$ iff every $\self{f}{X}$
with a point of period $q$ also has a point of period $p$.
When $X$ is the $n$-od, this preorder is also a partial order
(\ie\ is antisymmetric) and is characterized in~\cite{b}
by ${\le_{X_n}}=\bigcap_{t\le n}{\bo_t}$
where each $\bo_t$ is a partial ordering defined below.
Baldwin actually proves something stronger, that if
$\self{f}{X_n}$, then the set of $f$-periods is a finite
union of sets each a $\bo_t$-initial segment for some $t\le n$.

First, $\bo_1$ is the \sark\ linear ordering defined by
$2^i(2a+1)\bo_1 2^j(2b+1)$ iff
\begin{itemize}
\item $a=0=b$ and $i\le j$,
\item $a=0<b$,
\item $0<a,b$ and $i>j$, or
\item $0<b<a$ and $i=j$,
\end{itemize}
for all $a,b,i,j\ge 0$.
($(\naturals,{\bo_1})$ has order type $\om+(\om^*)^2$.)
Second, given $n>1$ and $m,k\ge 1$:
\[m\bo_n k\Leftrightarrow
\begin{cases}
  \tcase k=1:& m=1\\
  \tcase n\divides k:& m=1\tor n\divides m\tand m/n\bo_1 k/n\\
  \tcase n\notdivides k\not=1:& m\in\{1,k\}\cup\{ik+jn:i\ge 0\tand j\ge 1\}\\
\end{cases}
\]
($(\naturals,{\bo_n})$ is a disjoint union of $n$ chains,
one chain of type $\om+(\om^*)^2$
below $n-1$ chains of type $\om^*$.)

Baldwin proves a result related to Corollary~\ref{fromperiod}.
To state it, we must first give his classification of
the finite orbits of a given $\self{f}{X_n}$ into \emph{types}.
If $o$ is in a finite orbit $O$ then $O$ has type $1$.
(Thus, any $f$ satisfying the hypotheses of
Theorem~\ref{allperiods} has an orbit of type $1$.)
On the other hand, if $o$ is not in $O$, then
$O$ has type $p$ for each period $p$ of
the partial map $\self{f_O}{[n]}$ where
$f_O(i)=j$ if $O\cap \beta_i$ is nonempty and
$f$ maps to $\beta_j$ the point in $O\cap \beta_i$ closest to $o$.
Baldwin proved that if $f$ has an orbit of size $k$
that has type $p$, then, for each $m\bo_p k$,
$f$ has a point of period $m$.
Since, for example, $x\bo_1 4\Leftrightarrow x\in\{1,2,4\}$,
the full periodicity of case $n=3$ of Corollary~\ref{fromperiod}
is not a corollary of Baldwin's type-based analysis.

\section{Proofs of Theorem \ref{allperiods} and \ref{scramble}}
\label{proofs}

\begin{definition}
Given $x,y\in X_n$, let the closed interval
$[x,y]$ denote the unique arc with endpoints $x$ and $y$.
Define open and half-open intervals as closed intervals
with appropriate points removed.
Given arcs $I, J$ of $X_n$ and $\self{g}{X_n}$,
we say that $I$ \emph{$g$-covers} $J$
and write $I\covers{g}J$ if $g(I)\supset J$.
\end{definition}

The next two propositions are fundamental properties
of star graphs that we will use without comment.

\begin{proposition}
  If $a,b\in X_n$ and $\self{g}{X_n}$,
  then $[a,b]\covers{g}[g(a),g(b)]$.
\end{proposition}
\begin{proof}
  $g([a,b])$ is connected and
  $[g(a),g(b)]$ is the smallest connected
  superset of $\{g(a),g(b)\}$.  
\end{proof}

\begin{definition}
  Given an arc $I\subset X_n$, a \emph{compatible ordering}
  of $I$ is a linear ordering of $I$ such that
  the order topology on $I$
  equals the subspace topology inherited from $X_n$.
\end{definition}

Order each branch $\closure{\beta_i}$ of $X_n$
by the unique compatible ordering $\le_i$ such that
$o=\min(\closure{\beta_i})$.
We will omit the subscript of $\le_i$ when safe to do so.

\begin{proposition}
If $a,b,c\in X_n$
and $o\not\in(a,b)$, then $x\in[y,z]$ for some
permutation $x,y,z$ of $a,b,c$.
\end{proposition}
\begin{proof}
  The points $a$ and $b$ must be on the same branch,
  and if $c$ is also on that branch,
  then the proposition is clearly true.
  If $c$ is not in the same branch as $a$ and $b$,
  then, letting $\{x\le y\}=\{a,b\}$, we have
  $[c,y]\supset [o,y]\supset [x,y]$.
\end{proof}

\begin{definition}
Given $\self{g}{X_n}$, by a \emph{$g$-cascade} we mean
a finite or infinite sequence of arcs
$I_0,I_1,I_2,\ldots$ such that for all $i\ge 1$ we have
$I_{i-1}\covers{g}I_i$ and $o\not\in I_i^\circ$
where $Y^\circ$ denotes the interior of $Y$.
By a \emph{$g$-loop} we mean a $g$-cascade $I_0,\ldots,I_m$
such that $I_m\supset I_0$.
\end{definition}

\begin{lemma}\label{shrink}
  If $I_0,I_1,I_2,\ldots$ is a $g$-cascade, then
  there is a descending chain of arcs
  $I_0=Q_0\supset Q_1\supset Q_2\supset\cdots$
  such that $g^i(Q_i)=I_i$ for all $i$.
\end{lemma}
\begin{proof}
  Construct $Q_0,Q_1,\ldots,Q_m,\ldots$ by recursion on $m$.
  Given $Q_{m-1}$, let $h=g^m$ and observe that
  $h(Q_{m-1})=g(I_{m-1})\supset I_m$.
  Choose $Q_m=[a,b]$ minimal among the subarcs of $Q_{m-1}$
  that $h$-cover $I_m$.
  Then $I_m=[h(a),h(b)]$ because if $I_m=[h(c),h(d)]$
  then $[c,d]$ is not a proper subinterval of $[a,b]$.
  Moreover, if $z\in(a,b)$ and $h(z)\not\in I_m$, then,
  since $o\not\in I_m^\circ$, there is a permutation $x,y$ of $a,b$
  such that $h(x)\in[h(y),h(z))$, which implies there is
  $w\in[y,z)$ such that $I_m=[h(y),h(w)]$
  in contradiction with the minimality of $Q_m$.
  Thus, $h(Q_m)=I_m$.\
\end{proof}

\begin{lemma}\label{loop}
  If $I_0,\ldots,I_m$ is a $g$-loop then
  for some $x\in I_0$ we have $g^m(x)=x$
  and $g^i(x)\in I_i$ for all $i$.
\end{lemma}
\begin{proof}
  Let $Q_0,\ldots,Q_m$ be as in Lemma~\ref{shrink}.
  Then $g^m(Q_m)=I_m\supset I_0\supset Q_m$.
  Since $I_m$ and $Q_m$ are arcs, we may assume
  that $I_m=[0,1]$ and $Q_m=[a,b]\subset[0,1]$.
  Applying the Intermediate Value Theorem,
  $g^m$ has a fixed point $x$ in $Q_m$. 
  Finally, $g^i(x)\in g^i(Q_m)\subset g^i(Q_i)=I_i$.
\end{proof}

We prove Lemma~\ref{basicperiods} below using the well-known
(see~\cite{b} for citations)
technique of analyzing the restriction of $\covers{g}$
to pairs of minimal elements of the set of
intervals with endpoints in a fixed $g$-orbit.

\begin{definition}
  Given $\self{g}{X_n}$, a \emph{$g$-basic interval} is a
  minimal element of the set of closed intervals of the
  form $[a,b]$ where $a$ and $b$ are distinct elements
  of the $g$-orbit of $o$.
\end{definition}

In~\cite{b}, Baldwin defines ``basic intervals''
as above but assumes $g(o)=o$ and replaces
the orbit of $o$ with the union of $\{o\}$
and another fixed finite orbit.

\begin{lemma}\label{basicperiods}
  If $\self{g}{X_n}$, $m\ge 2$, and $B_{-1},B_0,B_1,B_2,\ldots, B_m$
  is a $g$-cascade of $g$-basic intervals
  such that $B_{-1}=B_0=B_m$ and
  $B_i\not=B_j$ for all $\{i<j\}\subset\{0,\ldots,m-1\}$,
  then, for all $p\ge m$, $g$ has a point of period $p$.
\end{lemma}
\begin{proof}
  Fix $p\ge m$ such that $p$ is not the period of $o$.
  Every sequence of the form
  $B_0,B_0,B_0,\ldots,B_0,B_1,B_2,B_3,\ldots,B_m$ is a $g$-loop.
  Therefore, by Lemma~\ref{loop},
  there exists $x\in B_0$ such that $g^p(x)=x$,
  $g^i(x)\in B_0$ for all $i\in[0,p-m]$, and
  $g^i(x)\in B_{i+m-p}$ for all $i\in[p-m,p]$.
  Let $q$ be the period of $x$.

  Seeking a contradiction, suppose that $q<p$.
  If $m-1\le q<p$, then $g^{p-1}(x)$ is in
  the orbit of $o$ because
  $B_{m-1}\ni g^{p-1}(x)=g^{p-q-1}(x)\in B_0$;
  if $q<m$, then $g^{p-q}(x)$ is in
  the orbit of $o$ because
  $B_{m-q}\ni g^{p-q}(x)=g^{p}(x)\in B_0$.
  Therefore, $x$ and $o$ have the same orbit.
  Since $B_0\not=B_1$, the orbit of $o$
  must have at least 3 points. 
  Therefore, $x$, $g(x)$, and $g^2(x)$ are 3 distinct points
  in the orbit of $o$ and so cannot all be endpoints of $B_0$.
  Therefore, $p-m\le 1$ and, hence, $q\le m$.

  For each basic interval $I$,
  $\max(I)$ is well-defined and not $o$.
  Moreover, $\max(I)\not=\max(J)$
  for all distinct basic intervals $I$ and $J$.
  Therefore, $m\le q-1$, in contradiction with $q\le m$.
\end{proof}

\begin{lemma}\label{genscramble}
  Suppose that $\self{g}{X_n}$, $u,v\in X_n$,
  $\le$ is a compatible ordering of $[g(u),g(v)]$
  such that $g(v)<u<v\le g(u)$, and $B_0,\ldots B_p$
  is a $g$-loop such that  $B_0=[u,v]$, $B_1\subset[g(v),u]$, 
  and $B_{p-1}$ is disjoint from $(u,v)$.
  Then $g$ is Li-Yorke chaotic.
\end{lemma}
\begin{proof}
  Inductively construct an infinite sequence
  $x_0,x_1,x_2,\ldots$ as follows.
  Let $x_0=v$ and choose $x_1\in[u,v)$ such that $g(x_1)=x_0$.
  Observe that $x_1\in(g(x_0),g(x_1))$.
  Inductively assume we have $m>0$,
  $g(x_m)=x_{m-1}$, and $x_m\in(g(x_{m-1}),g(x_m))$.
  Choose $x_{m+1}\in(x_{m-1},x_m)$ such that $g(x_{m+1})=x_m$.
  Now we have $x_{m+1}\in(x_{m-1},x_m)=(g(x_m),g(x_{m+1}))$;
  hence, the inductive hypotheses have been preserved.
  This completes the construction of $\vec{x}$.
  Next, observe that $x_1<x_0$ and $x_{m+1}\in(x_{m-1},x_m)$
  for all $m\ge 1$, so $x_1<x_3<x_5<\cdots<x_4<x_2<x_0$.
  Let $a=\lim x_{2i+1}$ and $b=\lim x_{2i}$;
  observe that $g(a)=b$ and $g(b)=a$.
  Let $A_{2i+1}=[x_{2i+1},a]$ and $A_{2i}=[b,x_{2i}]$.
  Since $g(a)=b<x_{2i}=g(x_{2i+1})$ for all $i\ge 0$
  and $g(x_{2i})=x_{2i-1}<x_{2i+1}<a=g(b)$ for all $i\ge 1$,
  we have $A_{j+1}\covers{g} A_j$ for all $j\ge 0$.

  We may assume $p$ is even,
  for we may replace $B_0,\ldots, B_p$ with
  $B_0,\ldots B_p,B_0,\ldots B_p$ without loss.
  For each real $r\in[0,1]$, choose $E_r\subset\naturals$
  with asymptotic density $r$ and define an infinite
  sequence $I_r(0),I_r(1),I_r(2),\ldots$ as the concatenation
  of the infinite sequence of finite sequences
  $\vec C_1,\vec D_1,\vec C_2,\vec D_2,\vec C_3,\vec D_3,\ldots$
  where
  \begin{align*}
    \vec C_k&=A_{2k},A_{2k-1},A_{2k-2},\ldots,A_0\\
    \vec D_k&=
    \begin{cases}
      A_{p-1},A_{p-2},A_{p-3},\ldots,A_1 &\text{if } k\in E_r \\
      B_1,B_2,B_3,\ldots,B_{p-1} &\text{if } k\not\in E_r.
    \end{cases}
  \end{align*} 
  This sequence is a $g$-cascade because:
  \begin{itemize}
  \item $A_{j+1}\covers{g} A_j$ for all $j\ge 0$.
  \item $B_j\covers{g} B_{j+1}$ for all $j<p$.
  \item $A_0=[b,v]\covers{g}[g(v),a]\supset B_1\cup A_{p-1}$.
  \item $A_1\covers{g}A_0\supset A_{2k}$.
  \item $B_{p-1}\covers{g}[u,v]\supset A_{2k}$.
  \end{itemize}
  Applying Lemma~\ref{shrink} (and compactness), choose
  $y_r\in I_r(0)$ such that $g^i(y_r)\in I_r(i)$ for all $i\ge 0$.
  
  Define a compatible metric $d$ on $X_n$ by requiring
  each branch to be isometric to $[0,1]$ and requiring
  $d(x,y)=d(x,o)+d(o,y)$ if $x$ and $y$ are on different branches.
  Since $o\not\in(u,v)$, we have
  $d(x,y)\ge\min_{w\in\{u,v\}}d(x,w)$
  for all $x\in(u,v)$ and $y\not\in(u,v)$.
  Let $\delta=\min\{d(u,b),d(b,v)\}$.
  Choose $\eps>0$ such that
  $d(x,y)<\eps$ implies $d(g(x),g(y))<\delta/2$
  for all $x\in X_n$ and $y\in\{u,v\}$.
  \begin{claim}
    Given $0\le r<s\le 1$,
    $d(g^i(y_r),g^i(y_s))\ge\eps$ infinitely often. 
  \end{claim}
  \begin{proof}
    Let $H=\{i:I_r(i)=B_{p-1}\text{ and }I_s(i)=A_1\}$,
    which is infinite.
    For each $i\in H$, we have $g^{i+1}(y_s)\in A_{2k}$ where
    $k$ is such that $i+1$ is the sum of the lengths of
    $\vec C_1,\vec D_1,\ldots,\vec C_{k-1},\vec D_{k-1}$.
    Hence, for all sufficiently large $i\in H$, we have
    \begin{align*}
      d(b,g^{i+1}(y_s))\le\delta/2
      &\Rightarrow\forall w\in\{u,v\}\ \,
      d(g^{i+1}(y_s),w)\ge\delta/2\\
      &\Rightarrow\forall z\in\{g(v),g(u)\}\ \,
      d(g^{i+1}(y_s),z)\ge\delta/2\\
      &\Rightarrow\forall w\in\{v,u\}\ \,
      d(g^i(y_s),w)\ge\eps\\
      &\Rightarrow d(g^i(y_s),g^i(y_r))\ge\eps\qedhere.
    \end{align*}
  \end{proof}
  Finally, since $\diam(A_k)\rightarrow 0$ as $k\rightarrow\infty$,
  \[\liminf_{i\rightarrow\infty}d(g^i(y_r),g^i(y_s))=0\]
  for all $r,s\in[0,1]$.
\end{proof}

\begin{proof}[Proof of Theorems~\ref{allperiods} and \ref{scramble}]
  Let $n\ge 2$, $\self{f}{X_n}$, $f(o)\in\beta_1$,
  and $f^3(o)\not\in\closure{\beta_1}$.
  There are three cases:
  \begin{enumerate}
  \item $f^2(o)\not\in\closure{\beta_1}$:
    let $u=o$, $v=f(o)$, and $B=[f^2(o),o]$.
  \item $o<_1 f(o)<_1 f^2(o)$:
    let $u=f(o)$, $v=f^2(o)$, and $B=[o,f(o)]$.
  \item $o<_1 f^2(o)<_1 f(o)$:
    let $u=f^2(o)$, $v=o$, and $B=[f(o),f^2(o)]$.
  \end{enumerate}
  In all three cases, let $A=[u,v]$ and verify
  that $A$ and $B$ are $f$-basic intervals,
  that $A\covers{f} A\covers{f} B\covers{f} A$,
  that $B\subset[f(v),u]$,
  and that $[f(u),f(v)]$ has a compatible ordering such that
  $f(v)<u<v\le f(u)$.
  By Lemmas~\ref{basicperiods} and \ref{genscramble},
  $f$ has points of all periods $\ge 2$ and is Li-Yorke chaotic.
  Since $X_n$ is a dendroid, $f$ also has a fixed point.
\end{proof}

\section{Orbits of $o$ of size $\ge n+2$}\label{extensions}
\begin{example}
  There exists $\self{f}{X_3}$ such that $o$ has
  period 5 and intersects every proper branch,
  but $f$ lacks period 3.
\end{example}
\begin{proof}
  Let $x_2=\max(\beta_2)$, $x_4=\max(\beta_3)$,
  and $o=x_0<x_1<x_3=\max(\beta_1)$.
  (See the diagram below.)
  Declare $f(x_i)=x_j$ where $j=i+1 \mod 5$.
  For convenience, we will write simply $i$ for $x_i$.
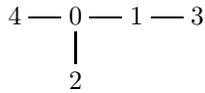
  \[\xymatrix@C=1pc@R=1pc{
    4\ar@{-}[r] & 0\ar@{-}[r] & 1\ar@{-}[r] & 3\\
    & 2\ar@{-}[u] &&
  }\]
  Then, for each minimal arc of the form $[i,j]$,
  extend $f$ to include a homeomorphism
  from $[i,j]$ to $[f(i),f(j)]$.
  To show that $f$ does not have period 3, we again
  use the method of analyzing the digraph $G$
  consisting of the restriction of $\covers{f}$ to
  pairs of $f$-basic intervals.
  $G$ is easily computed (see the diagram below),
  and its only 3-cycle is
  \[[0,1]\supset_f[0,1]\supset_f[0,1]\supset_f[0,1].\]
  \[\xymatrix@R=1pc{
    [0,1]\ar@(ul,dl)[] \ar[d] & [0,4]\ar[l]\\
    [0,2]\ar[r] & [1,3]\ar[l]\ar[u]
  }\]
  Seeking a contradiction, suppose $y\in X_3$ has period 3.
  Since the orbit of $y$ cannot intersect that of $o$,
  there exist $I_0,I_1,I_2,I_3$ in $G$ such that
  $f^i(y)\in I_i^\circ$ for all $i\le 3$.
  Moreover, $I_0,I_1,I_2,I_3$ must be an $f$-loop.
  Therefore, $0<f^i(y)<1$ for all $i$.
  But $f$ is order-reversing on $D=[0,1]\cap\inv{f}[0,1]$,
  so there are no orbits of size 3 in $[0,1]$.
\end{proof}

\begin{theorem}
  If $n\ge 3$, $\self{f}{X_n}$, and the orbit of $o$
  has size $n+2$ and intersects every proper branch,
  then $f$ is Li-Yorke chaotic and has all periods
  except possibly 3.
\end{theorem}
\begin{proof}
  We may assume $f(o)\in\beta_1$. 
  By Theorem~\ref{allperiods},
  we may assume also $f^3(o)\in\beta_1$.
  Therefore, the orbit of $o$ intersects $\beta_1$
  at exactly $f(o)$ and $f^3(o)$
  and intersects each other proper branch at exactly one point.
  In particular, $f^5(o)\not\in\beta_1$ and
  we may assume that $f^2(o)\in\beta_2$ and $f^4(o)\in\beta_3$. 
  There are two cases:
  \begin{enumerate}
  \item\label{threeinner} $o<_1 f^3(o)<_1 f(o)$:
    let $u=o$, $v=f^3(o)$ and $B_1=[o,f^4(o)]$.
  \item\label{threeouter} $o<_1 f(o)<_1 f^3(o)$:
    let $u=o$, $v=f(o)$, $B_1=[o,f^2(o)]$, $B_2=[f(o),f^3(o)]$,
    and $B_3=[o,f^4(o)]$.
  \end{enumerate}
  In both cases, let $A=[u,v]$.
  In Case~\ref{threeinner},
  $A\covers{f} A\covers{f} B_1\covers{f} A$.
  In Case~\ref{threeouter},
  $A\covers{f} A\covers{f} B_1\covers{f} B_2\covers{f} B_3\covers{f} A$.
  Therefore, by Lemma~\ref{basicperiods},
  $f$ has points of all periods $\ge 2$ in Case~\ref{threeinner}
  and points of all periods $\ge 4$ in Case~\ref{threeouter}.
  Since $X_n$ is a dendroid, $f$ also has a fixed point.
  Moreover, in Case~\ref{threeouter},
  $B_1\covers{f} B_2\covers{f} B_1$, which, by Lemma~\ref{loop},
  implies $x\in B_1$ such that $f(x)\in B_2$ and $f^2(x)=x$.
  Since $B_1$ and $B_2$ are disjoint, any such $x$ has period 2.

  In both Case 1 and Case 2, $B_1=[f(v),u]$ and
  $[f(u),f(v)]$ has a compatible ordering such that
  $f(v)<u<v\le f(u)$.
  By Lemma~\ref{genscramble}, $f$ is Li-Yorke chaotic.
\end{proof}

\begin{example}\label{justevens}
  There exists Li-Yorke chaotic $\self{f}{X_3}$ such that
  $o$ has period 6 and intersects every proper branch
  but the periodicity of $f$ is $\{1\}\cup 2\naturals$.
\end{example}
\begin{proof}
  Let $x_2=\max(\beta_2)$, $x_4=\max(\beta_3)$,
  and $o=x_0<x_1<x_3<x_5=\max(\beta_1)$.
  (See the diagram below.)
  Declare $f(x_i)=x_j$ where $j=i+1 \mod 5$.
  For convenience, we will write simply $i$ for $x_i$.
  \[\xymatrix@C=1pc@R=1pc{
    4\ar@{-}[r] & 0\ar@{-}[r] & 1\ar@{-}[r] & 3\ar@{-}[r] & 5\\
    & 2\ar@{-}[u] &&&
  }\]
  Then, for each minimal arc of the form $[i,j]$,
  extend $f$ to include a homeomorphism
  from $[i,j]$ to $[f(i),f(j)]$.
  Like in our previous example,
  to show that a given $y\in X_3$ does not
  have a given odd period $p\ge 3$, we analyze the digraph $G$:
  \[\xymatrix@R=1pc{
    [0,1]\ar@(ul,dl)[] \ar[r] & [0,2]\ar[r] & \ar[l] [1,3]\ar[r] & \ar[l] [0,4]\ar[r] & \ar[l] [3,5]
  }\]
  Since the orbit of $y$ cannot intersect that of $o$,
  there is an $f$-loop $I_0,\ldots,I_p$
  of elements of $G$ such that
  $f^i(y)\in I_i^\circ$ for all $i\le p$.
  But all odd cycles of $G$ are of the form
  $[0,1],\ldots,[0,1]$, so $0<f^i(y)<1$ for all $i$.
  But $f$ is order-reversing on $D=[0,1]\cap\inv{f}[0,1]$,
  so there are no odd orbits in $[0,1]$ except fixed points.

  It now suffices to show that $g=f^2$
  is Li-Yorke chaotic and has full periodicity.
  By Lemmas~\ref{basicperiods} and \ref{genscramble},
  this is indeed the case:
  letting $u=0$, $v=2$, $A=[u,v]$, and $B=[4,0]$,
  we have $A\covers{g} A\covers{g} B\covers{g} A$,
  $B=[g(v),u]$ and $g(u)=v$.
\end{proof}

\section{Open problems}
We should not be surprised that
$\le_{X_m}$ is weaker than $\le_{X_n}$ for $m\le n$
because, choosing a retraction
$r\colon X_n\rightarrow X_m$ and letting
$\iota\colon X_m\rightarrow X_n$ be its unique right inverse,
we have, for all $\self{g}{X_m}$ and $p\ge 0$,
that $\self{\iota\circ g\circ r}{X_n}$ and
$(\iota\circ g\circ r)^p=\iota\circ g^p\circ r$.
On the other hand, it is natural to wonder if
other interesting weakenings $\le_{\mc{F}}$ of $\le_{X_n}$
can be found by restricting to various sets $\mc{F}$ of maps
$\self{f}{X_n}$ not of the form $\iota\circ g\circ r$ above.
(To be precise, $p\le_\mc{F} q$ means that
every $f\in\mc{F}$ with a point of period $q$
also has a point of period $p$.)

An obvious candidate for $\mc{F}$ is the set
$\mc{T}_n$ of $\self{f}{X_n}$
with an orbit intersecting every proper branch.

\begin{problem}
  Characterize $\le_{\mc{T}_n}$.
\end{problem}

\begin{problem}
  If $f\in\mc{T}_n$ is witnessed by 
  the orbit of $o$ intersecting every proper branch,
  then what does the period of $o$ imply about the
  set of all periods of $f$?
\end{problem}

Theorem~\ref{allperiods} can be interpreted as a
modest partial solution to these problems.
Moreover, conjectured answers to the second problem
can be tested computationally if we limit
the size of the orbit of $f$ to, say, at most $10$.
Then an exhaustive computer search for when the conditions
of Lemma~\ref{basicperiods} are satisfied
by an interand of $f$ becomes quite feasible.
It would also then be feasible to automate
a search for absent digraph cycle lengths
like in the examples of section~\ref{extensions}.
In fact, the $f$ of Example~\ref{justevens}
is in one of only 24 classes of $f\in\mc{T}_3$ where
a 6-point orbit of $o$ hits every branch
yet Theorem~\ref{allperiods} does not apply.
Manual analysis of 24 digraphs shows that the
periodicity is always cofinite or $\{1\}\cup 2\naturals$.

We are also interested in proving Li-Yorke chaos
from larger orbits of $o$.

\begin{problem}
  If $f\in\mc{T}_n$ is witnessed by the orbit of $o$
  intersecting every proper branch, and the orbit
  of $o$ has cardinality in $[n+3,\infty)$,
  then is $f$ Li-Yorke chaotic?
\end{problem}

For small orbit sizes, we can exhaustively search
for small iterands $g$ of $f$
and $g$-loops $B_0,\ldots, B_p$ that
satisfy the hypotheses of Lemma~\ref{genscramble} where
the endpoints of $B_0,\ldots,B_p$ come from the orbit of $f$.
For orbits of $o$ size 6 in $X_3$ that hit every branch,
there are only 24 cases not covered by Theorem~\ref{scramble}.
Manual analysis reveals that Lemma~\ref{genscramble}
applies to $f$ or to $f^2$ in every case.


\begin{thebibliography}{9}
\bibitem{am}\textsc{Ll. Alsed\`a; J. M. Moreno}.
  Linear Orderings and the Full Periodicity Kernel for the $n$-Star.
  \emph{Journal of Mathematical Analysis and Applications}
  \textbf{180:2} (1993), 599--616.
\bibitem{b}\textsc{S. Baldwin.}
  An extension of \sark's Theorem to the $n$-od.
  \emph{Ergodic Theory and Dynamical Systems}
  \textbf{11:2} (1991), 249--271.
\bibitem{ly}\textsc{T. Li; J. Yorke.}
  Period Three Implies Chaos.
  \emph{The American Mathematical Monthly}
  \textbf{82:10} (1975), 985--992.
\end{thebibliography}
\end{document}